\documentclass[11pt,a4paper]{article}
\usepackage[latin1]{inputenc}
\usepackage{amsmath}
\usepackage{amsfonts}
\usepackage{amssymb}
\usepackage{graphicx}
\usepackage[boxruled]{algorithm2e}
\usepackage{enumerate}
\usepackage{fullpage}
\usepackage{siunitx}
\usepackage{multirow}
\usepackage{rotating}
\usepackage{fancyvrb}
\usepackage[colorlinks=true,urlcolor=blue,citecolor=blue]{hyperref}
\usepackage{hhline}
\usepackage{longtable}
\usepackage{appendix}
\usepackage{pgf}
\usepackage{adjustbox}

\ProvideTextCommand{\DJ}{OT1}{\raisebox{0.25ex}{-}\kern-0.4em D}
\newcommand{\setto}{\rightrightarrows}
\newcommand{\Fix}{\operatorname{Fix}}

\usepackage{amsthm}
\usepackage{thmtools}
\usepackage[capitalise]{cleveref}
\usepackage{xr}
\newtheorem{theorem}{Theorem}[section]
\newtheorem{proposition}[theorem]{Proposition}

\newtheorem{definition}[theorem]{Definition}
\newtheorem{fact}[theorem]{Fact}

\declaretheorem[style=remark,qed=$\Diamond$,Refname={Remark,Remarks},sibling=theorem]{remark}
\declaretheorem[style=remark,qed=$\Diamond$,Refname={Example,Examples},sibling=theorem]{example}

\Crefname{assumption}{Assumption}{Assumptions}
\newtheorem{formulation}[theorem]{Formulation}
\Crefname{formulation}{Formulation}{Formulations}

\title{A feasibility approach for constructing combinatorial designs of circulant type}
\author{Francisco J. Arag\'on Artacho\thanks{Department of Mathematics, University of Alicante, \textsc{Spain}. Email:~\href{mailto:francisco.aragon@ua.es}{francisco.aragon@ua.es}}
	   \and
	   Rub\'en Campoy\thanks{Department of Mathematics, University of Alicante, \textsc{Spain}.
	   	                     Email:~\href{mailto:ruben.campoy@ua.es}{ruben.campoy@ua.es}}
	   \and
	   Ilias Kotsireas\thanks{CARGO Lab,
	   	                     Wilfrid Laurier University, \textsc{Canada}.
                        	 Email:~\href{mailto:ikotsire@wlu.ca}{ikotsire@wlu.ca}}
	   \and
	   Matthew K. Tam\thanks{Inst.\ for Num.\ and Appl.\ Math.,
	   	                     University of G\"ottingen, \textsc{Germany}.
	   	                     Email:~\href{mailto:m.tam@math.uni-goettingen.de}{m.tam@math.uni-goettingen.de}}
	   }

\begin{document}

\maketitle

\begin{center}
  \emph{Dedicated to Jonathan M. Borwein}
\end{center}

\begin{abstract}
	In this work, we propose an optimization approach for constructing various classes of circulant combinatorial designs that can be defined in terms of autocorrelations. The problem is formulated as a so-called feasibility problem having three sets, to which the Douglas--Rachford projection algorithm is applied. The approach is illustrated on three different classes of circulant combinatorial designs: circulant weighing matrices, D-optimal matrices, and Hadamard matrices with two circulant cores. Furthermore, we explicitly construct two new circulant weighing matrices, a $CW(126,64)$ and a $CW(198,100)$, whose existence was previously marked as unresolved in the most recent version of Strassler's table.
\end{abstract}



\section{Introduction}
The notion of \emph{autocorrelation} associated with a finite sequence is a unifying concept that allows several classes of \emph{combinatorial designs of circulant type} to be concisely described. Designs of this type can be represented in terms of circulant matrices formed from finite sequences whose autocorrelation coefficients satisfy certain constancy properties; such sequences are called \emph{complementary sequences}. Examples of these designs include certain \emph{D-optimal matrices}, \emph{Hadamard matrices} and \emph{circulant weighing matrices} amongst  many other possibilities. A precise summary describing several of these designs, the associated sequences and their autocorrelation properties, can be found in \cite[Table~1]{kotsireas2013algorithms}. For an encyclopedic reference on autocorrelation properties and complementary sequences more generally, see \cite{Seberry:2017,SeberryYamada}, and for an authoritative reference on combinatorial designs, see \cite{HandbookCD}.

Many combinatorial designs can be defined as matrices of a given class which attain certain determinantal bounds. For instance, D-optimal and Hadamard matrices of a given order are precisely the $\{ \pm 1 \}$-matrices whose determinant is maximal among all other such matrices of the same order \cite{Brent,Horadam:book,Stinson:book}. For this reason, combinatorial designs arise in various fields where the determinantal bounds give rise to  ``best possible" or ``optimal" objects. Specific applications include coding theory \cite{Arasu:Aaron,book:Sala:Sakata}, quantum computing \cite{Flammia,WimVanDam}, wireless communication, cryptography and radar \cite{book:Golomb:Gong}. In many such applications, precise knowledge of the relevant combinatorial design is required.

In order to explicitly construct combinatorial designs of non-trivial orders, it is necessary to exploit underlying structure. Some possibilities include an appropriate group theoretic structure through which the mathematical analysis can proceed, or an efficient representation which is amenable to search algorithms such as metaheuristics. In this paper, we consider a novel approach closer in spirit to the latter. More precisely, we purpose the \emph{Douglas--Rachford algorithm (DRA)} from continuous (convex) optimization as a search heuristic. In this context, the DRA is a deterministic algorithm which traverses the combinatorial search space and which can be described in terms of a fixed-point iteration built from \emph{nearest point projection operators}. The critical feature of the DRA, which allows for its efficient implementation in this context, is that the autocorrelation function gives rise to a projection operator which can be efficiently computed. Although we will not touch further on it here, we note that there are several other works which consider application of the DRA to combinatorial problems, both from theoretical and experimental perspectives; see \cite{artacho2014recent,artacho2014douglas,artacho2016global,artacho2016solving,bauschke2017finite,elser2007searching,gravel2008divide}.

The remainder of this paper is organized as follows. In Section~\ref{sec:preliminaries}, we recall the necessary background regarding the Douglas--Rachford algorithm as well as other key results needed in the sequel. In Section~\ref{sec:model}, we give our \emph{feasibility problem} model for general combinatorial designs of circulant type before specializing it to \emph{circulant weighing matrices}, \emph{D-optimal designs of circulant-type} and \emph{double circulant core Hadamard matrices}. Finally, in Section~\ref{sec:results}, we provide computation results to illustrate the potential of the approach. In addition, in Theorem~\ref{th:cw existence}, we provide two circulant weighing matrices, a $CW(126,64)$ and a $CW(198,100)$, whose existence was unresolved in the latest version of \emph{Strassler's table} \cite[Appendix~A]{Tan-Strassler}.

\section{Preliminaries}\label{sec:preliminaries}

\subsection{The Douglas--Rachford algorithm}
Let $C_1,\,C_2,\dots,C_m$ be a collection of closed subsets in $\mathbb{R}^n$ with nonempty intersection. The corresponding \emph{($m$-set) feasibility problem} is
\begin{equation}\label{eq:feasibility problem}
\text{find~}x\in \bigcap_{j=1}^mC_j.
\end{equation}
Any feasibility problem of the form \eqref{eq:feasibility problem} can always be reformulated using Pierra's \emph{product-space reformulation}~\cite{Pierra} as an equivalent two set feasibility problem in the product Hilbert space $(\mathbb{R}^n)^m:=\mathbb{R}^n\times\stackrel{(m)}{\cdots}\times\mathbb{R}^n$. More precisely, the equivalence can be stated as
\begin{equation}\label{eq:2set feasibility problem}
x\in \bigcap_{j=1}^mC_j\subseteq\mathbb{R}^n \iff (x,x,\dots,x)\in C\cap D\subseteq(\mathbb{R}^n)^m:=\mathbb{R}^n\times\stackrel{(m)}{\cdots}\times\mathbb{R}^n,
\end{equation}
where the constraints $C$ and $D$, both subsets of $(\mathbb{R}^n)^m$, are defined to be
\begin{equation}\label{eq:product sets}
\begin{aligned}
C &:= \left\{(c_1,c_2,\dots,c_m):c_j\in C_j,\,j=1,2,\dots,m\right\}, \\
D &:= \left\{(x,x,\dots,x):x\in\mathbb{R}^n\right\}.
\end{aligned}
\end{equation}

The \emph{Douglas--Rachford algorithm} is an iterative method designed to solve two set feasibility problems (\emph{i.e.,}~\eqref{eq:feasibility problem} with $m=2$) and thus the equivalence in \eqref{eq:2set feasibility problem} is crucial for its application to finitely many-set problems.
 Given two subsets $A$ and $B$ of a Hilbert space $\mathcal{H}$, the algorithm can be compactly described as the fixed point iteration corresponding to the set-valued operator $T_{A,B}:\mathcal{H}\setto \mathcal{H}$ defined by
$$T_{A,B}:=\frac{I+R_BR_A}{2}=I+P_BR_A-P_A,$$
where $P_A:\mathcal{H}\setto A$ denotes the (potentially set-valued) \emph{projector} onto $A$ defined by
$$ P_A(x):=\left\{a\in A:\|x-a\|\leq\|x-a'\|,\forall\, a'\in A\right\},$$
and $R_A:=2P_A-I$ denotes the \emph{reflector} with respect to $A$. In other words, given an initial point $x_0\in \mathcal{H}$, the algorithm defines a sequence $(x_n)_{n=0}^\infty$ according to
\begin{equation}\label{eq:DR iteration}
x_{n+1}\in T_{A,B}(x_n)=\left\{x_n+b_n-a_n\in \mathcal{H}:a_n\in P_A(x_n),\,b_n\in P_B(2a_n-x_n)\right\}.
\end{equation}
We remark that the set-valuedness of $T_{A,B}$ arises from the fact that nearest points to a non-convex set need not be unique. In fact, the \emph{Motzkin--Bunt theorem} states that (in finite dimensions) the class of sets having everywhere unique nearest points are precisely those which are nonempty, closed and convex \cite[Theorem~9.2.5]{BorweinLewis}.

In order to apply the Douglas--Rachford algorithm it is necessary to have an efficient method for computing the projectors onto the individual sets. For the product-space feasibility problem specified by \eqref{eq:product sets}, this is the case whenever the projectors onto the underlying constraint sets in~\eqref{eq:feasibility problem} can be efficiently computed. This is summarized in the following proposition.
\begin{proposition}[Product-space projectors {\cite[Proposition~3.1]{artacho2014recent}}]
Suppose that $C_1,\dots,C_m$ are nonempty and closed subsets of $\mathbb{R}^n$. The projectors onto the sets $C$ and $D$ in \eqref{eq:product sets} are given, respectively, by
  $$ P_C((x_j)_{i=1}^m)=P_{C_1}(x_1)\times P_{C_2}(x_2)\times\dots\times P_{C_m}(x_m)\quad\text{and}\quad P_D((x_j)_{i=1}^m) = \left(\frac{1}{m}\sum_{i=1}^mx_i\right)_{j=1}^m. $$
\end{proposition}

Having discussed implementability of the DRA, we now turn our attention to its behavior. Our first observation is the following corespondence between \emph{fixed points} of the operator $T_{A,B}$ and points in $A\cap B$.
\begin{fact}[Fixed points of $T_{A,B}$]
  Let $A$ and $B$ be nonempty closed subsets of $\mathcal{H}$. If $x\in\Fix T_{A,B}:=\{x\in \mathcal{H}:x\in T_{A,B}(x)\}$, then there is a point $a\in P_A(x)$ such that $a\in A\cap B$.
\end{fact}
\begin{proof}
  Indeed, if $x\in T_{A,B}x$ then \eqref{eq:DR iteration} shows that there exist $a\in P_A(x)$ and $b\in P_B(2a-x)$ such that $x=x+b-a$. From the definition of the projectors onto $A$ and $B$, it follows that $a=b\in P_A(x)\cap P_B(2a-x)\subseteq A\cap B$ which proves the claim.
\end{proof}
 Consequently, if under appropriate condition, the DRA can be shown to converge to a fixed point, then it can be used to solve the feasibility problem. Unfortunately, for general combinatorial problems, there is no known unified framework which can be used to guarantee its convergence. Nevertheless, it is instructive to state the standard convergence results in the convex setting. Note that, in this case, both $A$ and $B$ have everywhere single-valued projectors, so the inclusions in \eqref{eq:DR iteration} can be replaced with equality.
\begin{fact}[Basic behavior of the Douglas--Rachford algorithm {\cite[Theorem~3.13]{BauComLuke}}]\label{fact:DRA}
	Let $A$ and $B$ be nonempty closed convex subsets of a finite dimensional Hilbert space $\mathcal{H}$. Let $x_0\in \mathcal{H}$ and define the sequence $(x_n)_{n=0}^\infty$ according to $x_{n+1}=T_{A,B}(x_n)$ for all $n\in\mathbb{N}$. Then either
	\begin{enumerate}[(i)]
		\item $A\cap B\neq\emptyset$ and $x_n\to x\in\Fix T$ with $P_A(x)\in A\cap B$, or
		\item $A\cap B=\emptyset$ and $\|x_n\|\to+\infty$.
	\end{enumerate}
\end{fact}
While Fact~\ref{fact:DRA} clearly fails to hold when the sets $A$ and $B$ are combinatorial in nature (except in trivial cases), it nevertheless serves as a good template for the expected behavior of the algorithm in non-convex settings. In particular, we see that it is not the sequence $(x_n)_{n=1}^\infty$ itself which is of interest but rather its \emph{shadow}; the sequence $(P_A(x_n))_{n=1}^\infty$. Implementation of the method is discussed in Algorithm~\ref{alg:DR}. Since it is partly problem specific, we delay the discussion of the precise form of the stopping criteria used in Algorithm~\ref{alg:DR} until Section~\ref{sec:results}.

\begin{algorithm}[H]
	\caption{Implementation of the Douglas--Rachford algorithm.}\label{alg:DR}
	\KwIn{$x_0\in \mathcal{H}$}
	$n=0$\;
	$a_0 \in P_{A}(x_0)$\;
	\While{{\rm stopping criteria not satisfied}}{
		$b_{n} \in P_{B}(2a_n-x_n)$\;
		$x_{n+1} = x_n+b_n-a_{n}$\;
		$a_{n+1} \in P_{A}(x_{n+1})$\;
		$n = n+1$\;
	}
	\KwOut{$a_n\in \mathcal{H}$}
\end{algorithm}

\subsection{Correlation and complementary}
In this section we recall the definitions of the \emph{periodic correlation operator} and \emph{complementary sequences} before deriving a result which we use to formulate a necessary condition in the sequel. Before this, we start by recalling the \emph{Jacobi--Trudi identity}.

Consider a vector of $n$ variables denoted by $x=(x_1,x_2,\dots,x_{n})$. A polynomial is \emph{symmetric} if it is invariant under every permutation of its variables. The \emph{$k$-th elementary symmetric polynomial} of $x$, denoted $\sigma_k$, is defined by
\begin{align*}
\sigma_k (x):= \sum_{1\leq j_1<j_2<\dots<j_k\leq n}\left( \;\prod_{l=1}^kx_{j_l} \right),
\end{align*}
Every symmetric polynomial can be written uniquely as a polynomial in the elementary symmetric polynomials (see~\cite[Theorem~1.1.1]{paule2008algorithms}). The \emph{$k$-th power polynomial} of $x$, denoted $p_k$, is defined by
\begin{align}\label{eq:p_k}
p_k (x):= \sum_{j=1}^nx_j^k.
\end{align}
The relationship between the latter two objects is provided by the following identity.
\begin{fact}[Jacobi--Trudi identity {\cite[p.~7]{paule2008algorithms}}]\label{th:jacobi trudi}
	For each $k=1,2,\dots,n$, it holds that
	\begin{equation}
	\sigma_k = \frac{1}{k!}\det \begin{pmatrix}
	p_1 & 1   & 0 & \dots & 0 \\
	p_2 & p_1 & 2 & \dots & 0 \\
	\vdots & \vdots & \ddots & \ddots & \vdots \\
	p_{k-1} & p_{k-2} & \dots & p_1 & k-1 \\
	p_k & p_{k-1} & \dots & \dots & p_1 \\
	\end{pmatrix}.
	\end{equation}
\end{fact}
The most important case of this identity for our purposes arises when $k=2$, in which case it yields
$$2\sigma_2=\det \begin{pmatrix}
p_1 & 1   \\
p_2 & p_1 \\
\end{pmatrix} = p_1^2-p_2.$$

 Let $\star\colon\mathbb{R}^n\times\mathbb{R}^n\to\mathbb{R}^n$ denote the \emph{periodic correlation operator} whose $s$-th entry is defined according to
 \begin{equation}\label{eq:periodic correlation}
 \left(a\star b\right)_s = \sum_{l=0}^{n-1}a_lb_{l+s}, \quad s=0,1,\ldots,n-1;
 \end{equation}
 where $a=(a_0,a_1,\dots,a_{n-1})\in\mathbb{R}^n$ and $b=(b_0,b_1,\dots,b_{n-1})\in\mathbb{R}^n$ are $n$-dimensional real vectors, and the indices in \eqref{eq:periodic correlation} understood modulo $n$.

 \begin{definition}[(Real) complementary sequences]
 	Consider vectors $a^0,a^1,\dots,a^{m-1}\in\mathbb{R}^n$. We say that the collection of sequences $\{a^j\}_{j=0}^{m-1}$ is \emph{(real) complementary} if there exist some constants $\nu_0$ and $\nu_1$ such that
 	   \begin{equation}\label{eq:complementary}
 	     \sum_{j=0}^{m-1}a^j\star a^j=(\nu_0,\nu_1,\dots,\nu_1).
 	   \end{equation}
 \end{definition}
 We note that the previous definition appears in \cite[Definition~2]{dokovic2015compression} for sequences which are potentially complex-valued.
 Using the Jacobi--Trudi identity, we are able to deduce the following necessary condition for complementary sequences which shall be used in the next section.
 \begin{proposition}[A necessary condition for complementary sequences]\label{prop:necessary condition for complenetary sequences}
   Suppose that the collection of sequences $\{a^j\}_{j=0}^{m-1}\subset\mathbb{R}^n$ is complementary with
       	   \begin{equation*}
       	   \sum_{j=0}^{m-1}a^j\star a^j=(\nu_0,\nu_1,\dots,\nu_1),
       	   \end{equation*}
   for constants $\nu_0$ and $\nu_1$. Then $\{p_1(a^j)\}_{j=0}^{m-1}\subset\mathbb{R}$ satisfy the equation
   $$\sum_{j=0}^{m-1}p_1^2(a^j)=\nu_1(n-1)+\nu_0,$$
   where $p_1$ is given by~\eqref{eq:p_k}.
 \end{proposition}
 \begin{proof}
  Applying the Jacobi--Trudi identity (\Cref{th:jacobi trudi}), we deduce that
  \begin{equation*}
    \sum_{s=1}^{n-1}(a^j\star a^j)_s  = 2\sigma_2(a^j) = \det \begin{pmatrix}
    p_1(a^j) & 1   \\
    p_2(a^j) & p_1(a^j) \end{pmatrix}
    = p_1^2(a^j)-p_2(a^j),
  \end{equation*}
  for all $j\in\{0,1,\dots,m-1\}$. Consequently,
  \begin{align*}
  \nu_1(n-1)  &= \sum_{s=1}^{n-1}\sum_{j=0}^{m-1}(a^j\star a^j)_s
 	  = \sum_{j=0}^{m-1}\sum_{s=1}^{n-1}(a^j\star a^j)_s \\
 	  &= \sum_{j=0}^{m-1}p_1^2(a^j)-\sum_{j=0}^{m-1}p_2(a^j)
 	  = \sum_{j=0}^{m-1}p_1^2(a^j)-\nu_0.
 	\end{align*}
  The claimed result follows by a routine rearrangement, thus completing the proof.
 \end{proof}

\section{Modelling Framework}\label{sec:model}
 In this section we explain how to model a general combinatorial design of circulant type as a three-set feasibility problem. More precisely, we consider designs belonging to the following class.
 \begin{definition}[Design of circulant type]\label{def:general design}
 	Consider natural numbers $n,m\in\mathbb{N}$, vectors $\alpha\in\mathbb{R}^m$ and $v\in\mathbb{R}^n$, and let $\mathcal{A}\subset\mathbb{R}$ be finite and nonempty. A \emph{design of circulant type} of order $n$ with parameters $(m,\alpha,v,\mathcal{A})$ is an $m$-tuple of vectors,
 	 $$(a^0,a^1,\dots,a^{m-1})\in (\mathcal{A}^n)^m:=\mathcal{A}^n\times \stackrel{(m)}{\cdots}\times \mathcal{A}^n,$$
 	which satisfy the following two conditions:
    $$\sum_{s=0}^{n-1}a_s^j=\alpha_j\quad\forall j\in\{0,1,\dots,m-1\},\text{~~and~~} \sum_{j=0}^{m-1}a^j\star a^j=v.$$
 \end{definition}	
 We remark that the notation ``$\mathcal{A}$" will be reserved for a finite subset of $\mathbb{R}$ which we refer to as the \emph{alphabet}. In this work, we will be concerned with the alphabets $\{\pm1\}$ and $\{0,\pm 1\}$.

\begin{formulation}\label{form:main}
Let $\mathcal{A}\subset\mathbb{R}$ be finite and nonempty, and let $\alpha\in\mathbb{R}^m$ and $v\in\mathbb{R}^n$. Consider the feasibility problem
\begin{equation}\label{eq:formulation main equation}
   \text{find }(a^0,a^1,\dots,a^{m-1})\in C_1\cap C_2\cap C_3\subseteq (\mathbb{R}^n)^m,
\end{equation}
where the constraint sets are defined by
 \begin{subequations}\label{eq:vector formulation constraints}
 	\begin{align}
 	C_1  &:= \left\{(a^0,a^1,\dots,a^{m-1})\in (\mathbb{R}^n)^m: a^j\in\mathcal{A}^n,\,\forall j=0,1,\dots,m-1\right\}, \label{eq:C1} \\
 	C_2  &:= \left\{(a^0,a^1,\dots,a^{m-1})\in (\mathbb{R}^n)^m :\sum_{s=0}^{n-1}a_s^j=\alpha_j,\forall j=0,1,\dots,m-1\right\}, \label{eq:C2} \\
 	C_3  &:= \left\{(a^0,a^1,\dots,a^{m-1})\in (\mathbb{R}^n)^m:\sum_{j=0}^{m-1}a^j\star a^j=v\right\}. \label{eq:C3}
 	\end{align}
 \end{subequations}
\end{formulation}

\begin{remark}[Autocorrelation constraints in bit retrieval]
	In the special case that $m=1$, the constraint $C_3$ appears in the formulation of the \emph{bit retrieval} problem used in \cite{elser2007searching}.
\end{remark}

\begin{remark}[Variants of $C_1$]
  Within our framework, the constraint set $C_1$ in \eqref{eq:C1} can be easily modified so that the alphabet $\mathcal{A}$ set is different for each vector $a^j$ or even for each individual entries of the vectors $a^j$. In this way, desired entries of a design can be fixed or avoided by choosing the corresponding alphabet sets to be singleton or to exclude certain values, respectively.
\end{remark}

For each set of parameters $(m,\alpha,v,\mathcal{A})$, it transpires that an $m$-tuple of vectors $(a^j)_{j=0}^{m-1}$ satisfies  Definition~\ref{def:general design} precisely when it is a feasible point for Formulation~\ref{form:main}. This equivalence is justified by the following proposition.
\begin{proposition}\label{prop:equivalence of formulations}
  Let $\mathcal{A}\subset\mathbb{R}$ be nonempty and finite. A collection of real complementary sequences $\{a^j\}_{j=0}^{m-1}\subseteq \mathcal{A}^n$ satisfies \eqref{eq:complementary} with constants $\nu_0$ and $\nu_1$ if and only if $(a^j)_{j=0}^{m-1}\in(\mathbb{R}^n)^m$ solves \eqref{eq:formulation main equation} in \Cref{form:main} with $v=(\nu_0,\nu_1,\dots,\nu_1)$ and some $\alpha\in\mathbb{R}^m$ which satisfies
    $$\sum_{j=0}^{m-1}\alpha_j^2=\nu_1(n-1)+\nu_0.$$
\end{proposition}
\begin{proof}
	This is an immediate consequence of Proposition~\ref{prop:necessary condition for complenetary sequences}.
\end{proof}

In order for the feasibility problem defined by Formulation~\ref{form:main} to be computationally useful, it is necessary that the projectors onto the constraint sets in \eqref{eq:vector formulation constraints} can be efficiently computed. In what follows, we prove that this is indeed the case.

\begin{proposition}[Projector onto $C_1$]\label{prop:proj C1}
	Let $(a^0,a^1,\dots,a^{m-1})\in (\mathbb{R}^n)^m$. Then $P_{C_1}\left((a^{j})_{j=0}^{m-1}\right)$ is the set of points $(\bar{a}^{j})_{j=0}^{m-1}\in (\mathbb{R}^n)^m$ which satisfy, for all $j=0,1,\ldots,m-1$ and $s=0,1,\ldots,n-1$,
	\begin{equation}
	\bar{a}^j_s \in \left\{l\in\mathcal{A}:\left|l-a^j_s\right|=\min_{\bar{l}\in\mathcal{A}}\left|\bar{l}-a^j_s\right|\right\}.
	\end{equation}
\end{proposition}	
\begin{proof}
  Let $a\in\mathbb{R}$. We observe that projector onto the set $\mathcal{A}$ is given by
  $$P_{\mathcal{A}}(a) = \left\{l\in\mathcal{A}:|l-a|=\min_{\bar{l}\in\mathcal{A}}\left|\bar{l}-a\right|\right\}.$$
  Applying this result pointwise and using definition of the inner-product on $(\mathbb{R}^n)^m$, the result follows.
\end{proof}

\begin{proposition}[Projector onto $C_2$]\label{prop:proj C2}
	Let $(a^0,a^1,\dots,a^{m-1})\in (\mathbb{R}^n)^m$ and $e=(1,1,\dots,1)\in\mathbb{R}^n$. Then
	$$ P_{C_2}\left((a^{j})_{j=0}^{m-1}\right) = \left(a^j+\frac{1}{n}\left(\alpha_j-\sum_{s=0}^{n-1}a^j_s\right)e\right)_{j=0}^{m-1}.$$
\end{proposition}
\begin{proof}
	The projection of any point $a\in\mathbb{R}^n$ onto the hyperplane $H_j:=\left\{a\in\mathbb{R}^n:e^T a = \alpha_j\right\}$ is given by (see, for instance, \cite[Example~3.21]{BauCom})
	$$ P_{H_j}(a)= a+\frac{1}{\|e\|^2}\left(\alpha_j-e^T a\right)e = a+\frac{1}{n}\left(\alpha_j-\sum_{s=0}^{n-1}a_s\right)e. $$
	The definition of the inner-product on $(\mathbb{R}^n)^m$ yields $P_{C_2}\left((a^{j})_{j=0}^{m-1}\right)=\left(P_{H_j}(a^{j})\right)_{j=0}^{m-1}$, from which the result follows.
\end{proof}

Thus, implementation of the projectors given in Proposition~\ref{prop:proj C1}~\&~\ref{prop:proj C2}, requires only vector arithmetic and finding the minimum of a finite set. From a computation perspective, the latter poses no problem when the alphabet, $\mathcal{A}$, is small. We now turn our attention to describing the projector onto $C_3$.

Let $\mathcal{F}:\mathbb{C}^n\to\mathbb{C}^n$ denote the  \emph{(unitary) discrete Fourier transform (DFT)}, that is, the linear mapping defined for any $a\in\mathbb{C}^n$ by
$$\mathcal{F}(a):=\frac{1}{\sqrt{n}}\begin{pmatrix}
\,1 & 1 & \cdots& 1\\
\,1 & \omega^{1\cdot 1}&\cdots &\omega^{1\cdot(n-1)}\\
\,\vdots &\vdots&\ddots&\vdots\\
\,1 & \omega^{(n-1)\cdot 1}&\cdots&\omega^{(n-1)\cdot(n-1)}
\end{pmatrix}a,$$
where $\omega:=e^{2\pi i/n}$ is a primitive $n$-th root of unity. Let $\mathcal{F}^{-1}$ denote its inverse.
In the following facts, both $|\cdot|$ and $(\cdot)^2$ are understood in the pointwise sense, and $(\cdot)^\ast$ denotes the (complex) conjugate of a complex number.

 \begin{fact}[Properties of the DFT]\label{fact:DFT properties}
 	Let $a=(a_0,a_1,\ldots,a_{n-1})\in\mathbb{C}^{n}$.
 	\begin{enumerate}[(i)]
 		\item\label{item:real DFT} (Conjugate symmetry) $a\in\mathbb{R}^n$ if and only if
 		$\mathcal{F}(a)$ is \emph{conjugate symmetric}, that is,
 		$$\mathcal{F}(a)_0\in\mathbb{R} \text{~~and~~} \mathcal{F}(a)_s=\left(\mathcal{F}(a)_{n-s}\right)^*, \forall s=1,2,\ldots,n-1. $$
 		\item\label{item:DFT convulation} (Correlation theorem) $\mathcal{F}(a\star a)=|\mathcal{F}(a)|^2.$
 		\item\label{item:DFT linear isometry} $\mathcal{F}$ is a linear isometry on $\mathbb{C}^n$.
 	\end{enumerate}
 \end{fact}
 \begin{proof}
 	\eqref{item:real DFT}:~See, e.g.,~\cite[pp.~76--77]{BH95}.
 	\eqref{item:DFT convulation}:~See, e.g.,~\cite[p.~83]{BH95}.
 	\eqref{item:DFT linear isometry}:~This follows from the fact that $\mathcal{F}$ is unitary.
 \end{proof}

In the following proposition, we denote the unit sphere in $\mathbb{C}^m$ by $$\mathbb{S}:=\left\{(z_j)_{j=0}^{m-1}\in\mathbb{C}^m:\sum_{j=0}^{m-1}|z_j|^2=1\right\},$$
 and we set
 $Y:=\mathcal{F}(\mathbb{R}^n)^m=\mathcal{F}(\mathbb{R}^n)\times\stackrel{(m)}{\dots}\times\mathcal{F}(\mathbb{R}^n)$
 where, due to Fact~\ref{fact:DFT properties}, the set $\mathcal{F}(\mathbb{R}^n)$ is precisely the set of conjugate symmetric vectors in $\mathbb{C}^n$, i.e.,
   $$\mathcal{F}(\mathbb{R}^n)=\left\{(z_s)_{s=0}^{n-1}\in\mathbb{C}^n: z_0\in\mathbb{R},\, z_s=z_{n-s}^*,\forall s=1,2,\ldots,n-1 \right\}.$$

\begin{proposition}[Projector onto $C_3$]
	Let $(\hat{a}^j)_{j=0}^{m-1}\in Y$, $v\in\mathbb{R}^n$ and $\hat{v}:=\mathcal{F}(v)$. Then
	\begin{equation}\label{eq:PC3_1}
	P_{C_3}=(\mathcal{F}^{-1},\dots,\mathcal{F}^{-1})\circ P_{\widehat{C_3}}\circ(\mathcal{F},\dots,\mathcal{F}),
	\end{equation}
	where the set $\widehat{C_3}$ is given by
	$$\widehat{C_3}:=\left\{(\hat{a}^j)_{j=0}^{m-1}\in Y: \sum_{j=0}^{m-1}|\hat{a}^j|^2=\hat{v}\right\}$$%
and $P_{\widehat{C_3}}\left((\hat{a}^j)_{j=0}^{m-1}\right)$ is given by  the set of all points $(\bar{a}^j)_{j=0}^{m-1}\in Y$ which satisfy, for all $s=0,1,\ldots,n-1$:
	\begin{equation}\label{eq:proj_C3}
	\begin{cases}
	(\bar{a}_s^j)_{j=0}^{m-1} = \frac{\sqrt{\hat{v}_s}}{\sqrt{\sum_{j=0}^{m-1}|{\hat{a}}_s^j|^2}}(\hat{a}_s^j)_{j=0}^{m-1}, & \text{if }(\hat{a}_s^j)_{j=0}^{m-1}\neq 0_m, \\
	(\bar{a}_s^j)_{j=0}^{m-1} \in  \sqrt{\hat{v}_s}\mathbb{S}, & \text{if }(\hat{a}_s^j)_{j=0}^{m-1}= 0_m. \\
	\end{cases}
	\end{equation}	
\end{proposition}	
\begin{proof}
	We first prove the claimed formula for $P_{\widehat{C_3}}$. To this end, note that
	\begin{equation}\label{eq:decom_c3}
	\widehat{C_3} = E\cap Y\text{ where } E:= \left\{(\hat{a}^j)_{j=0}^{m-1}\in (\mathbb{C}^n)^m: \sum_{j=0}^{m-1}|\hat{a}^j|^2=\hat{v}\right\}.
	\end{equation}
	As the projector onto $\mathbb{S}$ for a point $z\in\mathbb{C}^m$ is given by
	\begin{equation}\label{eq:proj_S}
	P_{\mathbb{S}}(z) = \begin{cases}
	z/\|z\|, & \text{if }z\neq 0_m, \\
	\mathbb{S}, & \text{if }z=0_m, \\
	\end{cases}
	\end{equation}
	applying \eqref{eq:proj_S} to each $m$-tuple $(\hat{a}_s^j)_{j=0}^{m-1}$, we  deduce that $(\bar{a}^j)_{j=0}^{m-1}\in P_{E}\left((\hat{a}^j)_{j=0}^{m-1}\right)\subset(\mathbb{C}^n)^m$ precisely when the vector ${(\bar{a}_s^j)_{j=0}^{m-1}\in\mathbb{C}^m}$ satisfies~\eqref{eq:proj_C3}
for all $s=0,\ldots,n-1$.	Due to \eqref{eq:decom_c3}, any vector $(\bar{a}^j)_{j=0}^{m-1}$ which satisfies \eqref{eq:proj_C3}
and is contained in $Y$ is an element of $P_{\widehat{C_3}}\left((\hat{a}^j)_{j=0}^{m-1}\right)$. Thus the claimed formula for $P_{\widehat{C_3}}$ follows.
	
	Next we prove \eqref{eq:PC3_1}. We first note that since the Fourier transform, $\mathcal{F}$, is a linear isometry on $\mathbb{C}^n$ (Fact~\ref{fact:DFT properties}\eqref{item:DFT linear isometry}), the operator $(\mathcal{F},\dots,\mathcal{F})$ is a linear isometry on $(\mathbb{C}^n)^m$ with inverse given by $(\mathcal{F},\dots,\mathcal{F})^{-1}=(\mathcal{F}^{-1},\dots,\mathcal{F}^{-1})$. Thanks to \cite[Lemma~3.21]{HesseThesis}, we therefore have that
	\begin{equation}\label{eq:P_C3 complex}
	P_{C_3}=(\mathcal{F}^{-1},\dots,\mathcal{F}^{-1})\circ P_{\mathcal{F}(C_3)}\circ(\mathcal{F},\dots,\mathcal{F}),
	\end{equation}
	where $\mathcal{F}(C_3):=\left\{\left(\mathcal{F}(a^j)\right)_{j=0}^{m-1}:(a^j)_{j=0}^{m-1}\in C_3\right\}$. To complete the proof, it therefore suffices to show $\mathcal{F}(C_3)=\widehat{C_3}$.  To this end, we observe that for a tuple $(a^j)_{j=0}^{m-1}\in(\mathbb{C}^n)^m$, we have
	$$\sum_{j=0}^{m-1}a^j\star a^j=v ~\stackrel{\text{Fact~\ref{fact:DFT properties}\eqref{item:DFT linear isometry}}}{\iff}~ \sum_{j=0}^{m-1}\mathcal{F}(a^j\star a^j)=\hat{v}
	~\stackrel{\text{Fact~\ref{fact:DFT properties}\eqref{item:DFT convulation}}}{\iff}~ \sum_{j=0}^{m-1}\left|\mathcal{F}(a^j)\right|^2=\hat{v},$$
	which shows that $\mathcal{F}(C_3)\subseteq \widehat{C_3}$. To deduce the reverse inclusion, note that $\mathcal{F}$ is invertible (Fact~\ref{fact:DFT properties}\eqref{item:DFT linear isometry}) and use the same argument with $(a^j)_{j=0}^{m-1}:=(\mathcal{F}^{-1}(\hat{a}^j))_{j=0}^{m-1}$.
\end{proof}	

\begin{remark}[Equation~\eqref{eq:proj_C3}]
	We emphasize that it is important to note that the projector onto $\widehat{C_3}$ is given by \eqref{eq:proj_C3} for tuples $(\bar{a}^j)_{j=0}^{m-1}$ contained in $Y$ but not $(\mathbb{C}^n)^m$.
\end{remark}

We now provide three concrete examples of types of combinatorial designs which can be described in terms of the structure proposed in \Cref{form:main}.

\subsection{Circulant weighing matrices}\label{ssec:CW matricces}
Recall that a matrix $W\in\mathbb{R}^{n\times n}$ is said to be \emph{circulant} if there is a vector $w\in\mathbb{R}^n$ such that the rows of $W$ are cyclic permutations of $w$ (offset by their row index).
 \begin{definition}[Circulant weighing matrix]\label{def:cw matrices}
 	Let $n,k\in\mathbb{N}$. A \emph{circulant weighing matrix} of order $n$ and weight $k$, denoted $\operatorname{CW}(n,k^2)$, is a circulant matrix $W\in\{0,\pm1\}^{n\times n}$ such that
 	  \begin{equation}\label{eq:cw equality}
 	    WW^T=k^2I,
 	  \end{equation}
 where $I\in\mathbb{R}^{n\times n}$ denotes the identity matrix.
 \end{definition}
Since the matrix $W$ is circulant, there exists a vector $a\in\{0,\pm 1\}^n$ such that $W=c(a)$ where the mapping $c:\mathbb{R}^n\to\mathbb{R}^{n\times n}$ maps a vector to an associated circulant matrix.  For such a vector, the equality \eqref{eq:cw equality} is equivalent to
  \begin{equation}\label{eq:autocorrelation CW}
    a\star a=(k^2,0,0,\dots,0).
  \end{equation}
Applying \Cref{prop:equivalence of formulations} (with $m=1$), we therefore arrive at the following.
\begin{proposition}\label{prop:cw vector formulation}
   Let $n,k\in\mathbb{N}$. A matrix $W\in\mathbb{R}^{n\times n}$ is $CW(n,k^2)$ if and only if there exists a vector $a\in\{0,\pm1\}^n$ with $W=c(a)$ such that
\begin{enumerate}[(i)]
	\item\label{item:cw sum constraint} $\sum_{s=0}^{n-1}a_s=\pm k$, and
	\item $a\star a=(k^2,0,0,\dots,0)$.
\end{enumerate}
\end{proposition}

\begin{example}[A CW matrix of small order]
The vector $a=(-1, 1, 1, -1, 1, 0, 1, 0, 1, 1, 0, 0, -1)$ defines a $CW(13,3^2)$. Indeed, it verifies $\sum_{s=0}^{12}a_s=3$ and $a\star a=(9,0,0,0,0,0,0,0,0).$
\end{example}

The class of circulant weighing matrices are of interest, in part, because they include all \emph{circulant Hadamard matrices} (specially, a $CW(n,k^2)$ is a circulant Hadamard matrix whenever $n=k^2$ and $n=0\mod{4}$). The existence of a CW matrices for a given order and weight is, in general, not resolved. \emph{Strassler's table}, which originally appeared in 20 years ago in \cite{strassler}, gives the existence status of $CW(n,k^2)$ for $n\leq 200$ and $k\leq 10$. The table has been updated several times, but still contains open cases. The most up-to-date version known to the authors at the time of writing is contained in \cite[Appendix~A]{Tan-Strassler}. For other recent progress regarding CW matrices, see
\cite{TanThesis}. In Section~\ref{ssec:new designs} we solve two open cases by presenting two new circulant weighing matrices found with the DRA, namely, a $CW(126,8^2)$ and a $CW(198,10^2)$.

\subsection{D-optimal designs of circulant type}	
Let $n$ be an odd positive integer. Ehlich~\cite{E64} showed that the determinant of a square matrix of order $2n$ having $\{\pm 1\}$ entries satisfies the bound
$$|\det(D)|\leq 2^n(2n-1)(n-1)^{n-1}.$$
Such a matrix is said to be \emph{D-optimal} if it has maximal determinant, that is, the aforementioned determinate bound is attained. 

To construct a D-optimal matrix , it suffices to find two commuting square $\{\pm1\}$-matrices, $A$ and $B$, of order $n$ such that
\begin{equation}\label{eq:Dopt_matrix}
AA^T+BB^T=(2n-2)I+2J,
\end{equation}
where $J\in\mathbb{R}^{n\times n}$ denotes the matrix of all ones. A D-optimal matrix $D$ of order $2n$ can then be constructed from the matrices $A$ and $B$ as follows
\begin{equation}\label{eq:matrix_dopt}
D=\begin{pmatrix}
A&B\\
-B^T &A^T
\end{pmatrix}.
\end{equation}
This construction, originally proposed by Ehlich~\cite{E64} for the case in which $A$ and $B$ are circulant matrices, was later extended by Cohn~\cite{C89} to the setting in which the matrices commute. The former case constitutes a special type of D-optimal designs known as \emph{D-optimal designs of circulant type}.
\begin{definition}[D-optimal design of circulant type]
  A \emph{D-optimal design of circulant type} is a matrix $D$ of order $2n$ given by \eqref{eq:matrix_dopt} for a pair of circulant $\{\pm1\}$-matrices $A$ and $B$ of order $n$ satisfying~\eqref{eq:Dopt_matrix}.
  When we wish to refer to the underlying matrices $A$ and $B$ explicitly (rather than $D$), we shall say that $(A,B)$ is a D-optimal design of circulant type.
\end{definition}

Let $(A,B)$ be a D-optimal design of circulant type of order $2n$. As in the previous subsection, since both matrices $A$ and $B$ are circulant, there exist vectors $a,b\in\{\pm 1\}^n$ such that $A=c(a)$ and $B=c(b)$.  For such vectors, \eqref{eq:Dopt_matrix} is equivalent to
  \begin{equation}\label{eq:autocorrelation Dopt}
    a\star a+b\star b=(2n,2,2,\dots,2).
  \end{equation}
By applying \Cref{prop:equivalence of formulations} as before (now with $m=2$), we deduce the following characterization.
 \begin{proposition}\label{prop:dopt}
 	Let $n$ be an odd integer. A matrix $D$ is a D-optimal design of circulant type of order $2n$ if and only if there exist constants $\alpha,\beta\in\mathbb{Z}$ with $\alpha^2+\beta^2=4n-2$ and a pair of vectors ${(a,b)\in\{\pm 1\}^n\times \{\pm 1\}^n}$ such that $D$ satisfies~\eqref{eq:matrix_dopt} for $A=c(a)$ and $B=c(b)$, and the following assertions hold:
 	\begin{enumerate}[(i)]
 		\item $\sum_{s=0}^{n-1}a_s=\alpha$,
 		\item $\sum_{s=0}^{n-1}b_s=\beta$, and
 		\item $a\star a + b\star b=(2n,2,2,\dots,2)$.
 	\end{enumerate}
 \end{proposition}	
 \begin{example}[D-optimal design of circulant type of small order]
  The vectors $$a=(-1,1,-1,1,1,1,1,1,-1)\text{~~and~~}b=(-1,1,1,1,1,-1,1,1,1)$$ define a D-optimal design of order $9$. Let $\alpha=3$ and $\beta=5$. Then we have $\alpha^2+\beta^2=4n-2$ with
    $\sum_{s=0}^8a_s=\alpha$ and $\sum_{s=0}^8b_s=\beta,$
  and that $a\star a+b\star b = (18,2,2,2,2,2,2,2,2)$.
 \end{example}

 The existence of a D-optimal matrix for values $n < 100$ for which the Diophantine equation $x^2 +y^2 = 4n-2$ has solutions has been resolved in the affirmative with the exception of $n=99$; see \cite{dokovic2015dopt} and \cite[Table~1]{kotsireas2013newresults}. In other words, the first unresolved case of existence arises when $n=99$.

\subsection{Double circulant core Hadamard matrices}
Let $n$ be an odd positive integer. Recall that a \emph{Hadamard matrix} of order $n$ is a matrix $H\in\{\pm 1\}^{n\times n}$ such that $$HH^T=H^TH=nI.$$
There are many equivalent characterization of Hadamard matrices. For instance, they are precisely the $\{\pm 1\}$-matrices of maximal determinant \cite[Chapter~2]{Horadam:book}.

\begin{definition}[Double circulant core Hadamard matrix]\label{def:dcc hadamard matrices}
	Let $n\in\mathbb{N}$. A Hadamard matrix, $H$, of order $2n+2$ is said to be a \emph{Hadamard matrix with two circulant cores} if it is of either one of the following two forms
	 \begin{equation}\label{eq:dc circulant matrix}
	   \left(\begin{array}{cc|cccccc}
	     - & - & + & \dots & + & + & \dots & + \\
	     - & + & + & \dots & + & - & \dots & - \\ \hline
	     + & + &   &       &   &   &       &  \\
	     \vdots & \vdots & & A & & & B & \\
	     + & + &  & & & & & \\ \hline
	     + & - &  & & & & & \\
	     \vdots & \vdots & & B^T & & & -A^T & \\
	     + & - &  & & & & & \\
	   \end{array}\right),\qquad
	   \left(\begin{array}{cc|cccccc}
	     + & + &   &       &   &   &       &  \\
    	 \vdots & \vdots & & A & & & B & \\
	     + & + &  & & & & & \\ \hline
	     + & - &  & & & & & \\
	     \vdots & \vdots & & B^T & & & -A^T & \\
	     + & - &  & & & & & \\ \hline
	     - & - & + & \dots & + & + & \dots & + \\
	     - & + & + & \dots & + & - & \dots & - \\ 	
	   \end{array}\right),
	 \end{equation}
	where $A$ and $B$ are circulant matrices of order $n$, and $+$ and $-$ are shorthand for $+1$ and $-1$, respectively.
\end{definition}

We note that the two Hadamard matrices in~\eqref{eq:dc circulant matrix} are \emph{equivalent} in the sense that one can be obtained from the other via sequence of row/column negations and row/column permutation~\cite[\S2.1]{kotsireas2006hadamard}.

Two circulant matrices $A$ and $B$ satisfy \Cref{def:dcc hadamard matrices} precisely when \cite[p.~3]{kotsireas2006hadamard}
  \begin{equation}\label{eq:dcc hadamard matrix eq}
    AA^T+BB^T=(2n+2)I-2J.
  \end{equation}
Denote $A=c(a)$ and $B=c(b)$ for vectors $a,b\in\{\pm1\}^n$. It follows that \eqref{eq:dcc hadamard matrix eq} is equivalent to
   \begin{equation}
   \label{eq:auto cor dopt}
     a\star a+b\star b = (2n,-2,-2,\dots,-2).
   \end{equation}
Applying \Cref{prop:equivalence of formulations} as before, we deduce the following characterization.
 \begin{proposition}[Double circulant core Hadamard matrix]\label{prop:dcc hadamard matrices}
 	A pair of matrices $A$ and $B$ satisfy \eqref{eq:dcc hadamard matrix eq} and, consequently define a Hadamard matrix with two circulant cores, if and only if there exists vectors $a,b\in\{\pm 1\}^n$ such that $A=c(a), B=c(b)$ with
 	\begin{enumerate}[(i)]
 		\item\label{eq:DCHMi} $\sum_{s=0}^{n-1}a_s=\pm 1$
 		\item\label{eq:DCHMii} $\sum_{s=0}^{n-1}b_s=\pm 1$, and
 		\item\label{eq:DCHMiii} $a\star a + b\star b=(2n,-2,-2,\dots,-2)$.
 	\end{enumerate}
 \end{proposition}	
\begin{proof}
We note that as a direct consequence of \Cref{prop:equivalence of formulations}, one has $$\alpha_1^2+\alpha_2^2=\left(\sum_{s=0}^{n-1}a_s\right)^2+\left(\sum_{s=0}^{n-1}b_s\right)^2=2,$$
from which \eqref{eq:DCHMi}-\eqref{eq:DCHMii} follows. Condition \eqref{eq:DCHMiii} is the same as \eqref{eq:auto cor dopt} whose equivalence was already discussed before the statement of the proposition.
\end{proof}

 \begin{example}
 	The vectors $a=(1,-1,-1,1,-1,1,1,1,-1)$ and $b=(-1,-1,1,1,-1,1,1,1,-1)$ define a double core circulant Hadamard matrix design. Note that
 	$$\sum_{s=0}^8a_s=1, \qquad \sum_{s=0}^8b_s=1,$$
 	and $a\star a+b\star b = (18,-2,-2,-2,-2,-2,-2,-2,-2)$.
 \end{example}

\section{Computational Results}\label{sec:results}
 In this section, we report the results of numerical experiments which demonstrate the performance of DR feasibility formulation (DRA). In Section~\ref{ssec:new designs}, the formulation is used to construct two circulant weighing matrices, namely, a $CW(126,8^2)$ and a $CW(198,10^2)$. Until this work, the existence  of these matrices was an open question. We implement the DRA, described in Algorithm~\ref{alg:DR}, in \emph{Python~2.7} with the stopping criteria outlined in the following remark.

 \begin{remark}[stopping criteria]\label{r:stopping critera}
  Let $C:=C_1\times C_2\times C_3$ and $D$ denote the product space sets in~\eqref{eq:product sets} and let $\epsilon>0$ denote a small real number. Further, let $(x_n)$ denote a sequence generated by the DRA operator $T_{D,C}$. Denoting $p_n=(q_n,q_n,q_n):=P_D(x_n)$, we terminate the DRA when either a prespecified time limit is reached or  the following condition is satisfied:
       $$ \| (P_{C_1}, P_{C_1},P_{C_1})(p_n) - P_C(p_n) \|<\epsilon, $$
    where we note that $P_C=(P_{C_1},P_{C_2},P_{C_3})$. Notice that, if this condition is satisfied and $\epsilon\approx0$, then
      $$ (P_{C_1}, P_{C_1},P_{C_1})(p_n) \approx (P_{C_1},P_{C_2},P_{C_3})(p_n).$$
    In other words, we have $P_{C_1}(q_n)\approx P_{C_2}(q_n)\approx P_{C_3}(q_n)$, which implies $P_D(P_C(p_n))\approx P_C(p_n)$.
 \end{remark}

 \begin{figure}[htb!]

 	\centering
 	\scalebox{0.9}{\input{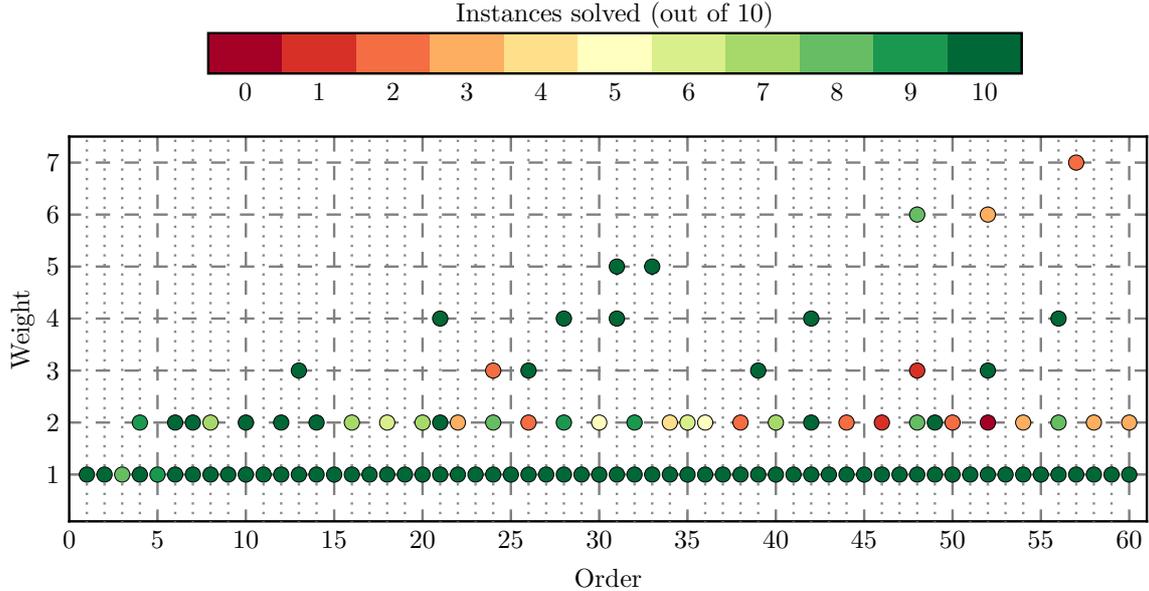}}
 	\caption{Results for CW matrices ($10$ random initialization, $3600$s time limit).\label{fig:CW results}}

 \end{figure}

 Once the stopping criteria in Remark~\ref{r:stopping critera} is satisfied, the resulting solution can be directly checked to see whether is conforms to Definition~\ref{def:general design}. Thus, whilst there may not be theory to guarantee that the DRA will convergence given enough time, if it does converge, then the question of  whether or not the output is a circulant design can be easily answered.

 Computational results for CW matrices are summarized in Figure~\ref{fig:CW results} and detailed computational results are included in Appendix~\ref{appendix:cw results}. Results for D-optimal designs of circulant-type and Hadamard matrices with two circulant cores, respectively, can be found in Table~\ref{t:Doptimal designs} and Table~\ref{t:two core hadamard}.

\begin{table}[htb!]
\centering
\caption{Experimental results for D-optimal designs ($10$ random initialization, $3600$s time limit).\label{t:Doptimal designs}}
\vspace{1ex}
\begin{tabular}{|c||S[table-format=2]S[table-format=3.2]S[table-format=7.1]|}
\hhline{-||---}
{Parameters} & {Solved instances} & {Average time (s)} & {Average iterations}\\
\hhline{=::===}
(3,1,3) 	&  10 &  0.00 &  3.4 \\
(5,3,3) 	&  10 &  0.00 &  6.6 \\
(7,1,5) 	&  9 &  0.01 &  12.7 \\
(9,3,5) 	&  10 &  0.19 &  398.3 \\
(13,1,7) 	&  7 &  0.13 &  349.7 \\
(13,5,5) 	&  7 &  0.16 &  403.6 \\
(15,3,7) 	&  10 &  0.24 &  591.8 \\
(19,5,7) 	&  10 &  0.81 &  1999.1 \\
(21,1,9) 	&  8 &  1.36 &  3424.9 \\
(23,3,9) 	&  8 &  2.02 &  5097.1 \\
(25,7,7) 	&  10 &  4.64 &  11668.6 \\
(27,5,9) 	&  9 &  116.50 &  297617.0 \\
(31,1,11) 	&  10 &  187.63 &  460501.0 \\
(33,3,11) 	&  8 &  553.44 &  1380160.0 \\
(33,7,9) 	&  8 &  810.97 &  2025880.0 \\
(37,5,11) 	&  3 &  1885.47 &  4399507.0 \\
(41,9,9) 	&  1 &  586.87 &  1352777.0 \\
(43,1,13) 	&  0 &  {--} &  {--} \\
(43,7,11) 	&  1 &  1207.20 &  2737865.0 \\
\hhline{-||---}
\end{tabular}
\end{table}

\begin{table}[htb!]
\centering
\caption{Experimental results for DHCM designs ($10$ random initialization, $3600$s time limit).\label{t:two core hadamard}}
\vspace{1ex}
\begin{tabular}{|c||S[table-format=2]S[table-format=3.2]S[table-format=7.1]|}
\hhline{-||---}
{Parameters} & {Solved instances} & {Average time (s)} & {Average iterations}\\
\hhline{=::===}
(1,1,1) 	&  10 &  0.00 &  1.7 \\
(3,1,1) 	&  10 &  0.01 &  33.6 \\
(5,1,1) 	&  10 &  0.00 &  5.9 \\
(7,1,1) 	&  8 &  0.01 &  35.8 \\
(9,1,1) 	&  10 &  0.01 &  35.2 \\
(11,1,1) 	&  10 &  0.04 &  89.2 \\
(13,1,1) 	&  9 &  0.10 &  222.2 \\
(15,1,1) 	&  10 &  0.10 &  241.8 \\
(17,1,1) 	&  10 &  0.22 &  549.3 \\
(19,1,1) 	&  10 &  1.68 &  4162.5 \\
(21,1,1) 	&  10 &  1.97 &  4764.0 \\
(23,1,1) 	&  10 &  2.26 &  5533.2 \\
(25,1,1) 	&  9 &  16.08 &  40468.1 \\
(27,1,1) 	&  10 &  76.10 &  192706.0 \\
(29,1,1) 	&  10 &  91.82 &  223875.0 \\
(31,1,1) 	&  10 &  428.61 &  1028850.0 \\
(33,1,1) 	&  10 &  849.84 &  2070120.0 \\
(35,1,1) 	&  4 &  2354.52 &  5864880.0 \\
(37,1,1) 	&  2 &  1883.67 &  4603068.0 \\
(39,1,1) 	&  1 &  2536.40 &  5916197.0 \\
\hhline{-||---}
\end{tabular}
\end{table}

\subsection{New circulant weighing matrices}\label{ssec:new designs}
 In this section, we state and prove our main result regarding the existence of two circulant weighing matrices. Our approach makes use of the following construction which is a consequence of~\cite[Theorem~2.3]{Arasu06}. Since this result appears without a proof in \cite[Section~2]{Arasu:Kotsireas:Koukouvinos:Seberry}, we show next how to derive it and give an explicit expression of the components of the constructed matrix in terms of the components of the original matrices.
\begin{theorem}\label{th:disjoint support theorem}
	Let $n,k\in\mathbb{N}$ with $n$ odd. Let $A$ and $B$ be two $CW(n,k^2)$ whose respective first rows, $a$ and $b$, have disjoint support\footnote{The \emph{support} of $c=(c_0,c_1,\ldots,c_{n-1})\in\mathbb{R}^n$ is the set $\{i\in\{0,\ldots,n-1\}: c_i\neq 0\}$.}. Then the circulant matrix $c(w)$ is a $CW(2n,4k^2)$ where the vector  $w=(w_0,w_1,w_2,\ldots,w_{2n-1})\in\mathbb{R}^{2n}$ is given component-wise by
\begin{equation}\label{eq:w_k}
w_s:=\left\{\begin{array}{ll}a_{\frac{s}{2}}+b_{\frac{s}{2}},&\text{if $s$ is even},\\
a_{\frac{s+n}{2}}-b_{\frac{s+n}{2}},&	\text{if $s$ is odd and $s\leq n-2$},\\
a_{\frac{s-n}{2}}-b_{\frac{s-n}{2}},&	\text{if $s$ is odd and $s> n-2$}.\end{array}\right.
\end{equation}

\end{theorem}
\begin{proof}
Let $G=\langle x\rangle=\left\{1,x,\ldots,x^{2n-1}\right\}$ be a cyclic group of order $2n$ generated by $x$, where $x^{2n}=1$. Clearly, the element $x^n$ of the group $G$ has order $2$.

Let $a,b\in\mathbb{R}^n$ denote the first rows of $A$ and $B$, respectively, and consider the generating functions given by
$$\mathbf{A}(x):=\sum_{s=0}^{n-1}a_sx^s\quad\text{and}\quad \mathbf{B}(x):=\sum_{s=0}^{n-1}b_sx^s.$$
Then $\mathbf{A}(x^2),\mathbf{B}(x^2)\in \mathbb{Z}[G]$ where $\mathbb{Z}[G]$ denotes the group ring of $G$ over $\mathbb{Z}$. Let $\widehat{a},\widehat{b}\in\mathbb{R}^{2n}$ denote the vectors associated with $\mathbf{A}(x^2)$ and $\mathbf{B}(x^2)$, respectively; that is,
\begin{equation}\label{eq:a_b_hat}
\widehat{a}:=(a_0,0,a_1,0,\ldots,a_{n-1},0)\quad\text{and}\quad\widehat{b}:=(b_0,0,b_1,0,\ldots,b_{n-1},0).
\end{equation}
Since $A$ and $B$ are $CW(n,k^2)$, a direct verification using Proposition~\ref{prop:cw vector formulation} shows that the circulant matrices defined by $\widehat{a}$ and $\widehat{b}$, namely $\widehat{A}:=c(\widehat{a})$ and $\widehat{B}:=c(\widehat{b})$, are $CW(2n,k^2)$.

Let $\widetilde{a},\widetilde{b}\in\mathbb{R}^n$ be the vectors associated with the formal sums $x^n\mathbf{A}(x^2)$ and $x^n\mathbf{B}(x^2)$, respectively. Since
$$x^n\mathbf{A}(x^2)=\sum_{s=0}^{n-1}a_sx^{2s+n}=\sum_{s=0}^{\frac{n-1}{2}}a_sx^{2s+n}+\sum_{s=\frac{n+1}{2}}^{n-1}a_sx^{2s-n},$$
it follows that
\begin{equation}\label{eq:a_tilde}
\widetilde{a}=\left(0,a_{\frac{n+1}{2}},0,a_{\frac{n+3}{2}},0,\ldots,a_{n-1},0,a_0,0,a_1,0,\ldots,0,a_{\frac{n-1}{2}}\right).
\end{equation}
The analogous expression holds for $\widetilde{b}$.

Consider now the circulant matrices $\widetilde{A}=c(\widetilde{a})$ and $\widetilde{B}=c(\widetilde{b})$ associated with the formal sums $x^n\mathbf{A}(x^2)$ and $x^n\mathbf{B}(x^2)$, respectively.
Since $a$ and $b$ have disjoint support and $n$ is odd, one can easily check that  $\widehat{a},\widetilde{a},\widehat{b},\widetilde{b}$ have pairwise disjoint support.
Therefore, all the assumptions of \cite[Theorem~2.3]{Arasu06} hold, and we deduce that the vector  $w\in\mathbb{R}^{2n}$ associated with the formal sum $\mathbf{W}(x)$ given by $$\mathbf{W}(x):=(1+x^n)\mathbf{A}(x^2)+(1-x^n)\mathbf{B}(x^2)\in\mathbb{Z}[G]$$
is such that the circulant matrix $c(w)$ is $CW(2n,4k^2)$.

To conclude the proof, we just need to check that the components of $w$ are given by~\eqref{eq:w_k}. Indeed, since
$$w=\widehat{a}+\widetilde{a}+\widehat{b}-\widetilde{b},$$
the expression given by~\eqref{eq:w_k} follows from~\eqref{eq:a_b_hat} and~\eqref{eq:a_tilde}.
\end{proof}

\begin{theorem}\label{th:cw existence}
	Both $CW(126,8^2)$ and $CW(198,10^2)$ exist.
\end{theorem}
\begin{proof}
Using the DRA, the following $CW(63,4^2)$ was found
\begin{Verbatim}[fontsize=\scriptsize]
a = [ 1,-1, 0, 0, 0, 0, 0, 1, 1, 0, 1, 0,-1, 0, 0,-1, 0, 0, 0, 0, 0, 0, 0, 0, 1, 0, 1,-1, 0, 0, 1, 0,
      0, 0,-1, 0, 0, 0, 0, 0,-1, 0, 0, 0, 0, 0, 0, 0, 0, 0, 0, 0, 0, 1, 0, 0, 0, 0, 1, 0, 0, 0, 1].
\end{Verbatim}
\noindent It has disjoint support with its cyclic permutation, $b$, given by
\begin{Verbatim}[fontsize=\scriptsize]
b = [ 0, 0, 0, 0, 0, 0, 0, 0, 0, 0, 0, 1, 0, 0, 0, 0, 1, 0, 0, 0, 1, 1,-1, 0, 0, 0, 0, 0, 1, 1, 0, 1,
      0,-1, 0, 0,-1, 0, 0, 0, 0, 0, 0, 0, 0, 1, 0, 1,-1, 0, 0, 1, 0, 0, 0,-1, 0, 0, 0, 0, 0,-1, 0].
\end{Verbatim}
\noindent The construction in Theorem~\ref{th:disjoint support theorem} applied to $a$ and $b$ yields
\begin{Verbatim}[fontsize=\scriptsize]
w = [ 1, 0,-1, 1, 0,-1, 0, 0, 0, 1, 0, 0, 0, 0, 1, 0, 1,-1, 0, 0, 1, 0, 1, 0,-1, 0, 0,-1, 0, 0,-1,-1,
      1, 1, 0, 0, 0, 0, 0,-1, 1, 0, 1, 1,-1, 0, 0, 1, 1, 0, 0, 0, 1, 1,-1, 0, 1, 0, 1, 1, 1, 1, 1, 1,
      0,-1,-1, 0,-1, 0, 0, 0,-1, 0, 0, 0, 0, 1, 0, 1,-1, 0, 0, 1, 0,-1, 0,-1, 0, 0, 1, 0, 0,-1, 1,-1,
     -1, 0, 0, 0, 0, 0, 1,-1, 0,-1, 1, 1, 0, 0,-1, 1, 0, 0, 0, 1, 1,-1, 0,-1, 0,-1,-1, 1, 1,-1],
\end{Verbatim}
and, consequently, the vector $w$ defines a $CW(126,8^2)$.

Similarly, using the DRA, the following $CW(99,5^2)$ was found
\begin{Verbatim}[fontsize=\scriptsize]
a = [-1, 0, 0, 1, 0, 0, 1, 0, 0, 1, 0, 0, 0, 0, 0,-1, 0, 0, 1, 0, 0, 0, 0, 0,-1, 0, 0, 1, 0, 0, 1, 0, 0,
     -1, 0, 0,-1, 0, 0, 0, 0, 0, 1, 0, 0, 1, 0, 0, 1, 0, 0, 1, 0, 0, 0, 0, 0, 1, 0, 0, 1, 0, 0,-1, 0, 0,
      1, 0, 0,-1, 0, 0, 1, 0, 0, 0, 0, 0, 1, 0, 0,-1, 0, 0, 0, 0, 0, 0, 0, 0,-1, 0, 0, 0, 0, 0,-1, 0, 0]
\end{Verbatim}
\noindent It has disjoint support with its cyclic permutation, $b$, given by
\begin{Verbatim}[fontsize=\scriptsize]
b = [ 0,-1, 0, 0, 1, 0, 0, 1, 0, 0, 1, 0, 0, 0, 0, 0,-1, 0, 0, 1, 0, 0, 0, 0, 0,-1, 0, 0, 1, 0, 0, 1, 0,
      0,-1, 0, 0,-1, 0, 0, 0, 0, 0, 1, 0, 0, 1, 0, 0, 1, 0, 0, 1, 0, 0, 0, 0, 0, 1, 0, 0, 1, 0, 0,-1, 0,
      0, 1, 0, 0,-1, 0, 0, 1, 0, 0, 0, 0, 0, 1, 0, 0,-1, 0, 0, 0, 0, 0, 0, 0, 0,-1, 0, 0, 0, 0, 0,-1, 0].
\end{Verbatim}
\noindent The construction in Theorem~\ref{th:disjoint support theorem} applied to $a$ and $b$ yields
\begin{Verbatim}[fontsize=\scriptsize]
w = [-1, 0,-1, 1, 0,-1, 1, 0, 1, 0, 0, 0, 1, 0, 1, 1, 0,-1, 1, 0, 1, 1, 0,-1, 0, 0, 0,-1, 0, 1,-1, 0,
     -1, 1, 0,-1, 1, 0, 1,-1, 0, 1, 0, 0, 0, 1, 0,-1,-1, 0,-1, 0, 0, 0, 1, 0, 1, 1, 0,-1, 1, 0, 1,-1,
      0, 1,-1, 0,-1, 0, 0, 0,-1, 0,-1, 0, 0, 0, 0, 0, 0,-1, 0, 1, 1, 0, 1, 0, 0, 0, 1, 0, 1,-1, 0, 1,
      1, 0, 1,-1, 0, 1, 1, 0, 1, 1, 0,-1, 0, 0, 0, 1, 0,-1, 1, 0, 1, 1, 0,-1, 1, 0, 1, 0, 0, 0,-1, 0,
     -1,-1, 0, 1, 1, 0, 1, 1, 0,-1,-1, 0,-1, 0, 0, 0, 1, 0, 1,-1, 0, 1, 0, 0, 0, 1, 0,-1, 1, 0, 1, 1,
      0,-1,-1, 0,-1,-1, 0, 1, 0, 0, 0,-1, 0, 1, 0, 0, 0, 0, 0, 0,-1, 0,-1, 1, 0,-1, 0, 0, 0, 1, 0,-1,
     -1, 0,-1, 1, 0,-1],
\end{Verbatim}
and, consequently, the vector $w$ defines a $CW(198,10^2)$.
\end{proof}

\begin{remark}\label{r:strassler table}
	Theorem~\ref{th:cw existence} resolves two open cases in the latest update of Strassler's Table appearing in the 2016 work of Tan~\cite[Appendix~A]{Tan-Strassler}. We also note that a previous version of Strassler's Table published in 2010 by Arasu~\&~Gutman~\cite[Table~3]{ArasuGutman} also listed these two cases as open. Despite the fact that these two cases have remained unresolved in multiple updates of Strassler's table, during the preparing of this manuscript (after independently proving Theorem~\ref{th:cw existence}) we discovered that existence can actually be deduced by combining either of the aforementioned versions of Strassler's table with a much older result of Arasu~\&~Dillon~\cite[Theorem~2.2]{ArasuDillon} which appeared in 1999. Specifically, the existence of $CW(126,8^2)$ and $CW(198,10^2)$ follows by respectively applying this result to $CW(21,4^2)$ and $CW(33,5^2)$, with $m=3$. In fact, the existence of $CW(198,10^2)$ was already claimed in~\cite{ArasuDillon}. This seems to have been missed until now.
\end{remark}

\begin{remark}
    Although Strassler's original table \cite{strassler} correctly states that $CW(196,4^2)$ exist, in both of updates, \cite[Appendix~A]{Tan-Strassler} and \cite[Table~3]{ArasuGutman}, its status is incorrectly shown as not existing. The same error appears in \cite[\S5]{gutmanthesis} and \cite[p.~144]{TanThesis}. Indeed, we obtained the following $CW(28,4^2)$ with the DRA
\begin{Verbatim}[fontsize=\scriptsize]
a = [1, 0, 1, -1, -1, 1, 0, 1, -1, 0, 0, 1, 0, 0, -1, 0, -1, -1, 1, 1, 0, 1, 1, 0, 0, 1, 0, 0],
\end{Verbatim}
from which a $CW(196,4^2)$ can be deduced by appending $0_6$ after each component
\begin{Verbatim}[fontsize=\scriptsize]
w = [ 1, 0, 0, 0, 0, 0, 0, 0, 0, 0, 0, 0, 0, 0, 1, 0, 0, 0, 0, 0, 0,-1, 0, 0, 0, 0, 0, 0,-1, 0, 0, 0,
      0, 0, 0, 1, 0, 0, 0, 0, 0, 0, 0, 0, 0, 0, 0, 0, 0, 1, 0, 0, 0, 0, 0, 0,-1, 0, 0, 0, 0, 0, 0, 0,
      0, 0, 0, 0, 0, 0, 0, 0, 0, 0, 0, 0, 0, 1, 0, 0, 0, 0, 0, 0, 0, 0, 0, 0, 0, 0, 0, 0, 0, 0, 0, 0,
      0, 0,-1, 0, 0, 0, 0, 0, 0, 0, 0, 0, 0, 0, 0, 0,-1, 0, 0, 0, 0, 0, 0,-1, 0, 0, 0, 0, 0, 0, 1, 0,
      0, 0, 0, 0, 0, 1, 0, 0, 0, 0, 0, 0, 0, 0, 0, 0, 0, 0, 0, 1, 0, 0, 0, 0, 0, 0, 1, 0, 0, 0, 0, 0,
      0, 0, 0, 0, 0, 0, 0, 0, 0, 0, 0, 0, 0, 0, 0, 1, 0, 0, 0, 0, 0, 0, 0, 0, 0, 0, 0, 0, 0, 0, 0, 0,
      0, 0, 0, 0],
\end{Verbatim}
since $196=28\cdot 7$.
\end{remark}

\paragraph{Acknowledgments.}
  This work is dedicated to the late Jonathan M. Borwein who suggested this project during his 2016 sabbatical in Canada.
  FJAA and RC were partially supported by MINECO of Spain  and  ERDF of EU, grant MTM2014-59179-C2-1-P. FJAA was supported by the Ram\'on y Cajal program by MINECO of Spain  and  ERDF of EU (RYC-2013-13327) and RC was supported by MINECO of Spain and ESF of EU (BES-2015-073360) under the program ``Ayudas para contratos predoctorales para la formaci\'on de doctores 2015''.
  IK is supported by an NSERC grant.
  MKT was supported by Deutsche Forschungsgemeinschaft RTG2088 and by a Postdoctoral Fellowship from the Alexander von Humboldt Foundation.

\clearpage

\begin{appendices}
\section{Detailed results for CW matrices}\label{appendix:cw results}
\begin{table}[!h]
\centering
\caption{Results for CW matrices ($10$ random initialization, $3600$s time limit).\label{t:CW}}
\begin{adjustbox}{totalheight=0.875\textheight}
\begin{tabular}{|c||S[table-format=2]S[table-format=3.2]S[table-format=7.1]|}
	\hhline{-||---}
	{$(n,k)$} & {No.\ Solved} & {Av.\ time (s)} & {Av.\ iterations}\\
	\hhline{=::===}
	(1,1)	&  10 &  0.00 &  1.5 \\
	(2,1)	&  10 &  0.00 &  1.4 \\
	(3,1)	&   8 &  0.00 &  3.1 \\
	(4,1)	&  10 &  0.00 &  5.6 \\
	(5,1)	&   9 &  0.00 &  4.0 \\
	(6,1)	&  10 &  0.00 &  4.1 \\
	(7,1)	&  10 &  0.00 &  3.3 \\
	(8,1)	&  10 &  0.00 &  3.5 \\
	(9,1)	&  10 &  0.00 &  4.0 \\
	(10,1)	&  10 &  0.00 &  4.5 \\
	(11,1)	&  10 &  0.00 &  4.0 \\
	(12,1)	&  10 &  0.00 &  3.8 \\
	(13,1)	&  10 &  0.00 &  4.7 \\
	(14,1)	&  10 &  0.00 &  3.8 \\
	(15,1)	&  10 &  0.00 &  5.7 \\
	(16,1)	&  10 &  0.00 &  6.0 \\
	(17,1)	&  10 &  0.00 &  5.7 \\
	(18,1)	&  10 &  0.00 &  4.6 \\
	(19,1)	&  10 &  0.00 &  7.0 \\
	(20,1)	&  10 &  0.00 &  6.2 \\
	(21,1)	&  10 &  0.00 &  6.3 \\
	(22,1)	&  10 &  0.00 &  8.2 \\
	(23,1)	&  10 &  0.00 &  6.9 \\
	(24,1)	&  10 &  0.00 &  6.2 \\
	(25,1)	&  10 &  0.00 &  4.8 \\
	(26,1)	&  10 &  0.00 &  5.2 \\
	(27,1)	&  10 &  0.00 &  6.1 \\
	(28,1)	&  10 &  0.00 &  6.7 \\
	(29,1)	&  10 &  0.00 &  8.7 \\
	(30,1)	&  10 &  0.00 &  7.9 \\
	(4,2)	&  9 &  0.00 &  5.1 \\
	(6,2)	&  10 &  0.00 &  8.1 \\
	(7,2)	&  10 &  0.09 &  328.9 \\
	(8,2)	&  7 &  0.02 &  81.4 \\
	(10,2)	&  10 &  0.04 &  180.5 \\
	(12,2)	&  10 &  0.05 &  211.7 \\
	(14,2)	&  10 &  0.16 &  649.2 \\
	(16,2)	&  7 &  0.09 &  373.0 \\
	(18,2)	&  6 &  0.03 &  110.5 \\
	(20,2)	&  7 &  0.30 &  1213.9 \\
	(21,2)	&  10 &  0.32 &  1165.9 \\
	(22,2)	&  3 &  0.02 &  67.7 \\
	(24,2)	&  8 &  0.13 &  506.5 \\
	(26,2)	&  2 &  0.03 &  101.0 \\
	(28,2)	&  9 &  0.19 &  703.3 \\
	(30,2)	&  5 &  0.07 &  232.6 \\
	(13,3)	&  10 &  0.05 &  172.0 \\
	(24,3)	&  2 &  10.93 &  42967.5 \\
	(26,3)	&  10 &  1.89 &  7162.3 \\
	(21,4)	&  10 &  11.47 &  45012.3 \\
	(28,4)	&  10 &  15.89 &  60377.3 \\
	(31,1)	&  10 &  0.00 &  9.0 \\	
	(32,1)	&  10 &  0.00 &  9.8 \\
	(33,1)	&  10 &  0.00 &  9.4 \\
	(34,1)	&  10 &  0.00 &  9.7 \\
	\hhline{-||---}	
\end{tabular}	
\quad
\begin{tabular}{|c||S[table-format=2]S[table-format=4.2]S[table-format=7.1]|}
	\hhline{-||---}
	{$(n,k)$} & {No.\ Solved} & {Av.\ time (s)} & {Av.\ iterations}\\
	\hhline{=::===}		
	(35,1)	&  10 &  0.00 &  8.5 \\
	(36,1)	&  10 &  0.00 &  8.3 \\
	(37,1)	&  10 &  0.00 &  11.7 \\
	(38,1)	&  10 &  0.00 &  6.2 \\
	(39,1)	&  10 &  0.00 &  9.9 \\
	(40,1)	&  10 &  0.00 &  10.5 \\
	(41,1)	&  10 &  0.00 &  11.8 \\
	(42,1)	&  10 &  0.00 &  11.8 \\
	(43,1)	&  10 &  0.00 &  9.1 \\
	(44,1)	&  10 &  0.00 &  8.7 \\
	(45,1)	&  10 &  0.00 &  9.7 \\
	(46,1)	&  10 &  0.00 &  14.5 \\
	(47,1)	&  10 &  0.00 &  9.3 \\
	(48,1)	&  10 &  0.00 &  10.9 \\
	(49,1)	&  10 &  0.00 &  11.9 \\
	(50,1)	&  10 &  0.00 &  13.4 \\
	(51,1)	&  10 &  0.00 &  11.7 \\
	(52,1)	&  10 &  0.00 &  16.3 \\
	(53,1)	&  10 &  0.01 &  17.8 \\
	(54,1)	&  10 &  0.00 &  16.2 \\
	(55,1)	&  10 &  0.00 &  14.7 \\
	(56,1)	&  10 &  0.00 &  10.4 \\
	(57,1)	&  10 &  0.00 &  15.9 \\
	(58,1)	&  10 &  0.00 &  11.6 \\
	(59,1)	&  10 &  0.00 &  12.4 \\
	(60,1)	&  10 &  0.00 &  16.1 \\
	(32,2)	&  9 &  0.25 &  984.3 \\
	(34,2)	&  4 &  0.06 &  211.8 \\
	(35,2)	&  6 &  0.14 &  516.3 \\
	(36,2)	&  5 &  0.09 &  359.4 \\
	(38,2)	&  2 &  0.11 &  398.0 \\
	(40,2)	&  7 &  0.34 &  1287.3 \\
	(42,2)	&  10 &  0.60 &  2265.0 \\
	(44,2)	&  2 &  0.06 &  241.0 \\
	(46,2)	&  1 &  0.02 &  65.0 \\
	(48,2)	&  8 &  0.21 &  798.0 \\
	(49,2)	&  10 &  1.36 &  5031.0 \\
	(50,2)	&  2 &  0.05 &  201.5 \\
	(52,2)	&  0 &  {-} &  {-} \\
	(54,2)	&  3 &  0.14 &  491.7 \\
	(56,2)	&  8 &  0.29 &  1098.4 \\
	(58,2)	&  3 &  0.01 &  44.3 \\
	(60,2)	&  3 &  0.28 &  1082.0 \\
	(39,3)	&  10 &  5.72 &  22158.7 \\
	(48,3)	&  1 &  13.29 &  52189.0 \\
	(52,3)	&  10 &  3.92 &  14888.2 \\
	(31,4)	&  10 &  422.45 &  1652410.0 \\
	(42,4)	&  10 &  132.12 &  504622.0 \\
	(56,4)	&  10 &  59.63 &  225106.0 \\
	(31,5)	&  10 &  23.10 &  90731.5 \\
	(33,5)	&  10 &  334.83 &  1306620.0 \\
	(48,6)	&  8 &  607.04 &  2365024.0 \\
	(52,6)	&  3 &  2314.49 &  8309650.0 \\
	(57,7)	&  2 &  482.54 &  1812060.0 \\
	\hhline{-||---}
    \multicolumn{4}{c}{}
\end{tabular}
\end{adjustbox}
\end{table}
\end{appendices}

\end{document}